\documentclass[12pt]{amsart}
\usepackage{mathtools}
\usepackage{tikz}
\usepackage{pgfplots}
\pgfplotsset{compat=1.18}
\usepgfplotslibrary{colormaps}
\usepackage{palatino}
\usepackage{xcolor}
\usepackage[colorlinks = true,
            linkcolor = blue,
            urlcolor  = blue,
            citecolor = blue,
            anchorcolor = blue]{hyperref}
\usepackage[margin=1.2in]{geometry}
\newtheorem{theorem}{Theorem}[section]
\newtheorem{lemma}[theorem]{Lemma}
\newtheorem{proposition}[theorem]{Proposition}
\newtheorem{corollary}[theorem]{Corollary}

\theoremstyle{definition}

\newtheorem{remark}[theorem]{Remark}

\newcommand{\B}{\mathcal{B}}

\renewcommand{\H}{\mathcal{H}}
\newcommand{\Aut}{\operatorname{Aut}}
\title[Geometric Properties of $\Aut(B)$]
{Operator Geometry of Hilbert Ball Automorphisms\\}

\author{Saikat Roy}

\address[Roy]{School of Advanced Sciences, VIT-AP University, Beside AP Secretariat, Amaravati, 522241, Andhra Pradesh, India.}
\email{saikatroy.cu@gmail.com, saikat.roy@vitap.ac.in}



\begin{document}

\hfill\\

\bigskip

\begin{abstract}
We consider the operator--theoretic model for the group of biholomorphic
automorphisms $Aut(B)$ of the unit ball $B$ of a complex Hilbert space $\H$ by representing each automorphism as a bounded linear operator on the augmented Hilbert space $\H\oplus \mathbb{C}$. Any member of $Aut(B)$ admits a natural block operator matrix representation acting on $\mathcal{H}\oplus\mathbb{C}$. We study the geometry of the subset
$M(\mathcal{H})$ of $\mathcal{B}(\mathcal{H}\oplus\mathbb{C})$ consisting
of these block operator matrices. It is shown that every element
corresponding to a non-rotation automorphism is a smooth point of
$\mathcal{B}(\mathcal{H}\oplus\mathbb{C})$. Orthogonality between two
such matrices is characterized geometrically by the antipodality of the
corresponding M\"{o}bius images of a boundary point of the ball. This
orthogonality characterization is applied to show that an inner
automorphism of $\Aut(B)$ that preserves Birkhoff--James orthogonality in both directions if and only if it is
conjugation by a pure rotation, yielding a rigidity result. The
normalized block matrices are $J$-unitary under a suitable normalization, where $J = \operatorname{diag}(I_{\mathcal{H}}, -1)$. We show that norm of such block matrices satisfy a submultiplicativity under a certain composition rule other than usual operator multiplication, and induce a metric on certain subsets of $Aut(B)$ which recover
the hyperbolic metric on the Hilbert ball. The symmetric
structure of Birkhoff--James orthogonality within $M(\mathcal{H})$ is
also studied: there are no left-symmetric points, while the only right-symmetric
points are pure rotations.
\end{abstract}

\medskip

\maketitle


\section{Introduction}

\noindent The automorphism group $\Aut(B)$ of the open unit ball $B$ of a complex
Hilbert space $\mathcal{H}$ is one of the central objects in infinite-dimensional complex function theory. Its elements are the M\"{o}bius transformations
\[
\phi_{a,U}(z) = U\frac{(P_a + \alpha(a)Q_a)z - a}{1 - \langle z, a \rangle},
\qquad z \in B,
\]
where $U$ is a unitary operator on $\mathcal{H}$, $a \in B$, $\alpha(a) = \sqrt{1 - \|a\|^2}$, and $P_a$, $Q_a$ denote the orthogonal projections onto $\mathbb{C}a$ and its orthogonal complement, respectively. Every $\phi_{a,U} \in \Aut(B)$ extends continuously to the closed unit ball $\overline{B}$, and the restriction to the unit sphere $\partial B$ is a homeomorphism of $\partial B$ onto itself \cite{Rudin, FV}.  The group $\Aut(B)$ plays a fundamental role in function theory, hyperbolic geometry, and the fixed point theory of holomorphic mappings \cite{Beardon-Short, Beardon-Wilker, Goebel-Reich}. The special elements $\phi_{0,U}$, corresponding to $a = 0$, are the \emph{pure rotations} and form a subgroup isomorphic to the group of unitary operators on $\H$.

\medskip

\noindent A natural operator--theoretic representation of $\Aut(B)$
arises by associating to each $\phi_{a,U}$ the block operator matrix
\[
M_{\phi_{a,U}} = \begin{bmatrix}
U(P_a + \alpha(a)Q_a) & -Ua \\
-a^* & 1
\end{bmatrix}
\]
acting on $\mathcal{H} \oplus \mathbb{C}$,  where the entries are understood as follows:
\begin{itemize}
\item $U(P_a + \alpha(a)Q_a) : \mathcal H \to \mathcal H$ is a bounded linear operator,
\item $-Ua : \mathbb C \to \mathcal H$ is the rank-one operator defined by
\[
\lambda \mapsto -\lambda\, Ua,
\]
\item $-a^* : \mathcal H \to \mathbb C$ is the bounded linear functional given by
\[
x \mapsto -\langle x, a \rangle,
\]
\item $1 : \mathbb C \to \mathbb C$ is the identity operator.
\end{itemize}
Accordingly, for $(x,\lambda) \in \mathcal H \oplus \mathbb C$,
\[
M_{\phi_{a,U}}(x,\lambda)
=
\left(
U(P_a + \alpha(a)Q_a)x - \lambda Ua,\;
-\langle x,a\rangle + \lambda
\right).
\]
In the special case whenever $U=I_\H$, the identity operator on $\H$, the block matrix $M_{\phi_{a,I}}$ associated to the automorphism $\phi_{a,I}$ is self-adjoint for any $a\in B$.
 
\subsection{Motivation and Main results} The assignment $\phi_{a,U} \mapsto M_{\phi_{a,U}}$ realizes $\Aut(B)$ as a subset $M(\H)$ of $\B(\mathcal{H} \oplus \mathbb{C})$ and opens a natural avenue for investigating the geometric properties of $\Aut(B)$ through operator space geometry, and this perspective is the central theme of the present article. Two geometric properties central to our investigation are smoothness and Birkhoff-James orthogonality. A nonzero element $x$ in a normed space $X$ is called a \emph{smooth point} if the norm on $X$ is G\^{a}teaux differentiable at $x$ \cite{James2, Roy-2024}. In a normed space $X$, an element $x$ is \emph{Birkhoff--James orthogonal} to $y$, written $x \perp_B y$, if $\|x + \lambda y\| \geq \|x\|$ for all scalars $\lambda \in \mathbb{C}$.

\medskip

\noindent We show that non-unitary points of $M \left(\H \right)$ corresponds smooth points in $B \left(\H\oplus \mathbb{C} \right)$ and this is one of the main results of this article.

\begin{theorem}\label{smooth: Aut}
Every element of $M \left(\H \right)$ which is not a pure rotation, represents a smooth point in $B \left(\H\oplus \mathbb{C} \right)$. 
\end{theorem}

\noindent Also, Birkhoff-James orthogonality in $M \left(\H \right)$ is equivalent with the antipodality of certain points.

\begin{theorem}\label{ortho: aut}
Let $a\in B$ be nonzero. Then $M_{\phi_{a,U}}\perp_B M_{\phi_{b,V}}$ in $M \left(\H \right)$ if and only if $\phi_{a,U} \left(-a/\|a\| \right)$ and $\phi_{b,V} \left(-a/\|a\| \right)$ are antipodal points in the unit sphere $\partial B$.
\end{theorem}

\noindent Unlike orthogonality in inner product spaces, Birkhoff--James orthogonality is not symmetric in general \cite{Sain-normattain}. An element $x$ is called \emph{left-symmetric} (respectively, \emph{right-symmetric}) if $x \perp_B y$ implies $y \perp_B x$ (respectively, $y \perp_B x$ implies $x \perp_B y$) for all $y \in X$. A point is called a symmetric point, if it is both left and right-symmetric. In the present paper, we say that $M_{\phi_{a,U}} \in M(\mathcal{H})$ is \emph{left-symmetric} (respectively, \emph{right-symmetric}) in $M(\mathcal{H})$ if $M_{\phi_{a,U}} \perp_B M_{\phi_{b,V}}$ implies $M_{\phi_{b,V}} \perp_B M_{\phi_{a,U}}$ (respectively, $M_{\phi_{b,V}} \perp_B M_{\phi_{a,U}}$ implies $M_{\phi_{a,U}} \perp_B M_{\phi_{b,V}}$) for all $M_{\phi_{b,V}} \in M(\mathcal{H})$,
where the Birkhoff--James orthogonality is considered in $\mathcal{B}(\mathcal{H} \oplus \mathbb{C})$. A point in $M(\H)$ is said to be symmetric, if it is left as well as right-symmetric point in $M(\H)$. Characterizing symmetric points in operator algebras has been an active direction of research \cite{Mal-Paul-Sain-book,Turnsek, Komuro-Saito-Tanaka}. However, the study of symmetricity in nonlinear subsets remain absent in the literature. We initiate this study in this article for the nonlinear subset $M(\H)$ of $B \left(\H\oplus \mathbb{C} \right)$.
\begin{theorem}\label{thm:left}
In $M(\mathcal{H})$, there are no left-symmetric points, and the right-symmetric points are precisely the pure rotations.
\end{theorem}

\noindent The orthogonality relation in $M(\H)$ naturally induces an orthogonality relation $\perp_{Aut}$ in $Aut(B)$, namely,
\[
\phi_{a,U}\perp_{Aut} \phi_{b,V} \quad \iff \quad M_{\phi_{a,U}}\perp_B M_{\phi_{b,V}}.
\]
We apply the preceding result to characterize inner automorphism of $Aut(B)$ that preserves Birkhoff-James orthogonality in both directions.
\begin{theorem}\label{thm: preserve}
Any inner automorphism in $Aut(B)$ preserves $\perp_{Aut}$ in both direction if and only if it is a conjugation by a pure rotation. 
\end{theorem}

\noindent As one of the important application of our result, we recover the hyperbolic metric through operator norm, and prove a submultiplicativity of norm which different that usual submultiplicativity of operator norm. For this we consider normalized block matrix representation $\widetilde{M}_{\phi_{a,U}}=\frac{1}{\alpha(a)}M_{\phi_{a,U}}$, and define fibre of a unitary operator $U\in B(\H)$, namely,
\[
F_{U} = \{\widetilde{M}_{\phi_{a,U}}:~a\in B\}.
\]
We define a multiplication $``\odot"$ in $F_{U}$ by
\[
\widetilde{M}_{\phi_{a,U}}\odot \widetilde{M}_{\phi_{b,U}} = \widetilde{M}_{\phi_{a,U}\circ \phi_{b,U}},
\]
where $``\circ"$ indicates the usual operation in $Aut(B)$. Consistent with the above terminologies, we prove the following:
\begin{theorem}\label{thm:hyp}
Let $U$ be any unitary operator on $\H$. Then
\begin{itemize}
    \item[(a)] $\|\widetilde{M}_{\phi_{a,U}}\odot \widetilde{M}_{\phi_{b,U}}\|\leq \|\widetilde{M}_{\phi_{a,U}}\|\|\widetilde{M}_{\phi_{b,U}}\|$.
    \item[(b)] The function $d_U:F_{U}\times F_{U}$ be defined as
\[
d_U(\widetilde{M}_{\phi_{a,U}},\widetilde{M}_{\phi_{b,U}}) = 2~ln~\|\widetilde{M}_{\phi_{a,U}}\odot \widetilde{M}_{\phi_{b,U}^{-1}}\|
\]
is a metric on $F_U$ and $(F_U,d_U)$ is isometric to $(B,d_{hyp})$, where $d_{hyp}$ denote the hyperbolic metric on $B$.
\end{itemize}
\end{theorem}

\medskip

\noindent The proofs of Theorems \ref{smooth: Aut} and \ref{ortho: aut} are included in Section~\ref{sec:proof}. In Section~\ref{section:preservation}, we discuss the symmetricity and preservation of orthogonality. Section~\ref{section:Hyperbolic metric} deals with the operator norm and hyperbolic metric. In Section~\ref{sec:prelim} collect the preliminary facts to be used in the sequel.

\medskip

\section{Preliminaries}\label{sec:prelim}

\noindent Let $X$ be a Banach space and let $A$ be a bounded linear operator on $X$. A sequence of unit vectors $(z_n)$ in $X$ is called a \emph{norming sequence} for $A$ if
\[
\lim_{n \to \infty} \|Az_n\| = \|A\|.
\]
Norming sequences are the key ingredient in the characterization of Birkhoff--James orthogonality of operators.

\begin{lemma}[\cite{Bhatia-Semrl}]\label{BS-Theorem}
For $A, T \in \B(\mathcal{H})$, $A \perp_B T$ if and only if 
there exists a norming sequence $(x_n)$ for $A$ such that 
$\lim_n \langle Ax_n, Tx_n \rangle = 0$.
\end{lemma}

\noindent Smoothness of operators on a Banach space can also be characterized in terms of norming sequences.

\begin{lemma}[\cite{Roy-2024}]\label{smoothness}
Let $A$ be a nonzero operator acting on a Banach space $X$, then $A$ is a smooth point in $B \left(X \right)$ if and only if the collection
\[
 \left\{\mu\in \mathbb{C}:~\lim_n x_n^* \left(Tx_n \right)=\mu,~ \left( (x_n,x_n^*) \right)\subset S_X\times S_{X^*},~\lim_n x_n^* \left(Ax_n \right)=\|A\| \right\}
\]
is a singleton set for any $T\in B \left(X \right)$.
\end{lemma}
\noindent If $X = \mathcal{H}$ is a Hilbert space, then identifying $x_n^* = \langle \cdot, y_n \rangle$ for each $n\in \mathbb{N}$, via the Riesz representation theorem, the preceding lemma takes the following form.

\begin{corollary}\label{smoothness-Hilbert}
Let $A$ be a nonzero operator acting on a Hilbert space $\H$. Then $A$ is a smooth point of 
$\mathcal{B}(\mathcal{H})$ if and only if the collection
\[
\left\{\mu \in \mathbb{C} : \lim_n \langle Tx_n, y_n \rangle 
= \mu,\ (x_n), (y_n) \subset S_{\mathcal{H}},\ 
\lim_n \langle Ax_n, y_n \rangle = \|A\|\right\}
\]
is a singleton set for every $T \in \mathcal{B}(\mathcal{H})$.
\end{corollary}

\noindent The following lemma is concerning about norm attainment and norming sequence of operators that admit a reducing subspace.
\begin{lemma}\label{block matrix norm}
Let $\mathcal{H}_0$ be a closed proper subspace of $\mathcal{H}$, and let $A \in \B(\mathcal{H})$ be such that $\mathcal{H}_0$ is a reducing subspace of $A$. Then 
\[
\|A\| = \max\left\{\|AP_0\|, \|A(I-P_0)\|\right\},
\]
where $P_0$ denotes the orthogonal projection onto $\mathcal{H}_0$. Moreover, if $\|A\| = \|AP_0\| > \|A(I-P_0)\|$, then $(z_n)$ is a norming sequence for $A$ if and only if $(z_n)$ is a norming sequence 
for $AP_0$.
\end{lemma}

\begin{proof}
Since $\H_0$ is a reducing subspace of $A$, let $T\oplus S$ be the block matrix representation of $A$ with respect to the orthogonal decomposition $\H_0\oplus \H_0^\perp$, where $T:=A|_{\H_0}:\H_0\to \H_0$ and $S:=A|_{\H_0^\perp}:\H_0^\perp\to \H_0^\perp$. For any unit vector $z \in \mathcal{H}$ we write $z = x + y$ with $x = P_0 z \in \mathcal{H}_0$ and $y = (I-P_0)z \in \mathcal{H}_0^{\perp}$. Since, $P_0$ commutes with $A$, we have
\[
\|Az\|^2 = \|AP_0 z\|^2 + \|A(I-P_0)z\|^2 
\leq \max\left\{\|AP_0\|^2, \|A(I-P_0)\|^2\right\}\|z\|^2.
\]
Since $\|A\| \geq \|AP_0\|$ and $\|A\| \geq \|A(I-P_0)\|$, we conclude that
\[
\|A\| = \max\left\{\|AP_0\|, \|A(I-P_0)\|\right\}.
\]

\medskip

\noindent Now suppose $\|A\| = \|AP_0\| > \|A(I-P_0)\|$, and let $(z_n)$ be a norming sequence for $A$. Write $z_n = x_n + y_n$ with $x_n \in \mathcal{H}_0$ and $y_n \in \mathcal{H}_0^{\perp}$, so that 
$\|x_n\|^2 + \|y_n\|^2 = 1$ for each $n$. Choose $\lambda$ with 
$\|AP_0\| > \lambda > \|A(I-P_0)\|$. Then
\begin{align*}
\|Az_n\|^2 &= \|AP_0 x_n\|^2 + \|A(I-P_0)y_n\|^2 \\
&\leq \|AP_0\|^2\|x_n\|^2 + \lambda^2\|y_n\|^2 \\
&= \|AP_0\|^2(\|x_n\|^2 + \|y_n\|^2) 
- (\|AP_0\|^2 - \lambda^2)\|y_n\|^2 \\
&= \|AP_0\|^2 - (\|AP_0\|^2 - \lambda^2)\|y_n\|^2.
\end{align*}
Since $\|Az_n\| \to \|A\| = \|AP_0\|$ and $\|AP_0\|^2 - \lambda^2 > 0$, 
it follows that $\|y_n\| \to 0$, and consequently $\|x_n\| \to 1$. 
Therefore,
\[
\lim_n \|AP_0 z_n\|^2 = \lim_n \|AP_0 x_n\|^2 = \lim_n \|Az_n\|^2 
= \|A\|^2 = \|AP_0\|^2,
\]
which shows $(z_n)$ is a norming sequence for $AP_0$. Conversely, if 
$(z_n)$ is a norming sequence for $AP_0$, then
\[
\|A\|^2 \geq \|Az_n\|^2 \geq \|AP_0 z_n\|^2 \to \|AP_0\|^2 = \|A\|^2,
\]
and hence $(z_n)$ is a norming sequence for $A$. This completes the 
proof.
\end{proof}

\noindent We also make note of the fact that any member $M_{\phi_{a,U}}$ of $M(\H)$ satisfies the $J$-contraction identity
\[
M_{\phi_{a,U}}^{\ast} J M_{\phi_{a,U}} = \alpha(a)^2 J,
\]
where $J = \operatorname{diag}(I_{\mathcal{H}}, -1)$ is the signature operator on
$\mathcal{H} \oplus \mathbb{C}$. The normalized matrix
$\widetilde{M}_{a,U} = \frac{1}{\alpha(a)} M_{\phi_{a,U}}$ therefore satisfies
$\widetilde{M}_{a,U}^{\ast} J \widetilde{M}_{a,U} = J$, making it a
$J$-unitary operator. Since $\alpha(a)$ is a positive scalar, this normalization
leaves the underlying M\"{o}bius map $\phi_{a,U}$ unchanged.

\section{Smoothness and orthogonality in $M(\H)$}\label{sec:proof}

\noindent In this section we prove that the the elements of $M(\mathcal{H})$ that are not pure rotations are smooth points in $\mathcal{B}(\mathcal{H}\oplus\mathbb{C})$, and translates Birkhoff--James orthogonality into antipodality on the unit sphere $\partial B$.

\subsection{Smoothness}\hfill\\

\noindent Pure rotations are never smooth points in 
$\mathcal{B}(\mathcal{H}\oplus\mathbb{C})$. Let 
$M_{0,U} = \mathrm{diag}(U,1)$ be a pure rotation. Since 
$M_{0,U}$ is unitary, $\|M_{0,U}\| = 1$ and every unit vector 
is a norming vector for $M_{0,U}$. Let $e_1, e_2 \in 
\mathcal{H}\oplus\mathbb{C}$ be any two orthonormal vectors, 
set $f_i = M_{0,U}e_i$ for $i = 1, 2$, and consider the 
rank-one operator
\[
T = f_1 \otimes e_1,
\]
that is, $T(x) = \langle x, e_1\rangle f_1$ for $x \in 
\mathcal{H}\oplus\mathbb{C}$. The constant pairs 
$(z_n, y_n) \equiv (e_i, f_i)$ satisfy 
$\langle M_{0,U}z_n, y_n\rangle = \langle f_i, f_i\rangle = 1 
= \|M_{0,U}\|$ for $i = 1, 2$, so both are norming pairs for 
$M_{0,U}$. A direct computation gives
\[
\langle Te_1, f_1\rangle = \langle f_1, f_1\rangle = 1, 
\qquad
\langle Te_2, f_2\rangle = \langle f_1, f_2\rangle 
= \langle M_{0,U}e_1, M_{0,U}e_2\rangle 
= \langle e_1, e_2\rangle = 0.
\]
Thus $\Lambda(M_{0,U}, T)$ contains both $0$ and $1$, and 
$M_{0,U}$ is not a smooth point in 
$\mathcal{B}(\mathcal{H}\oplus\mathbb{C})$ by 
Corollary~\ref{smoothness-Hilbert}. Therefore, we have proved the following proposition.

\begin{proposition}
Pure rotations in $M \left(\H \right)$ are not smooth points in $B \left(\H\oplus \mathbb{C} \right)$.
\end{proposition}

\noindent We now prove Theorem \ref{smooth: Aut}, which shows non-rotations are smooth points in $B \left(\H\oplus \mathbb{C} \right)$.

\begin{proof}[Proof of Theorem \ref{smooth: Aut}]
Consider $M_{\phi_{a,U}} \in M(\mathcal{H})$ for some nonzero 
$a \in B$ and unitary $U$ on $\mathcal{H}$. We write $M_{\phi_{a,U}} = \widetilde{U}M_{a,I},$, where $\widetilde{U} = \mathrm{diag}(U, 1) \in \mathcal{B}(\mathcal{H} \oplus \mathbb{C})$ is unitary. Since $\widetilde{U}$ is unitary, $\|M_{\phi_{a,U}}\| = \|M_{a,I}\|$, and 
$(z_n)$ is a norming sequence for $M_{\phi_{a,U}}$ if and only if it is 
a norming sequence for $M_{a,I}$. 

\medskip

\noindent It is not difficult to see that $\mathcal{H}_0 = \left\{(z, 0) \in \mathcal{H} \oplus \mathbb{C} : z \perp a\right\}$ is a subspace of co-dimension $2$ in $\mathcal{H} \oplus \mathbb{C}$, and is an eigenspace of $M_{a,I}$ corresponding the eigenvalue $\alpha(a)$. Since $M_{a,I}$ 
is self-adjoint, and $\H_0$ is an invariant subspace of $\H$, $\mathcal{H}_0$ is in fact 
a reducing subspace of $M_{a,I}$, giving the block decomposition $M_{a,I} = T\oplus S$ with respect to the orthogonal decomposition $\mathcal{H} \oplus \mathbb{C} = \mathcal{H}_0 \oplus \mathcal{H}_0^\perp$, where $T = \alpha(a)I$ 
on $\mathcal{H}_0$ and $S = M_{a,I}|_{\mathcal{H}_0^\perp}$.

\medskip

\noindent The complementary subspace $\mathcal{H}_0^\perp$ is 
a two-dimensional subspace, and spanned by the orthonormal basis 
$\left\{\frac{1}{\|a\|}(a, 0),\ (0,1)\right\}$. A direct computation shows that the matrix of $S$ with respect to this basis is given by
\[
[S] = \begin{bmatrix} 1 & -\|a\| \\ -\|a\| & 1 \end{bmatrix}.
\]
Since $[S]$ is self-adjoint, $\|S\|$ equals to the absolute value of its its largest eigenvalue, which is $1 + \|a\|$. The corresponding unit
eigenvector is 
\[
\frac{1}{\sqrt{2}}\left(\frac{1}{\|a\|}(a,0) - 
(0,1)\right) = \frac{1}{\sqrt{2}}\left(\frac{a}{\|a\|}, -1\right).
\]

\medskip

\noindent By Lemma~\ref{block matrix norm}, since 
$\|S\| = 1 + \|a\| > \alpha(a) = \|T\|$, we conclude
\[
\|M_{a,I}\| = \|S\| = 1 + \|a\|,
\]
and the norm is attained precisely on the unit sphere of the one-dimensional eigenspace:
\[
\left\{\mu\cdot\frac{1}{\sqrt{2}}
\left(\frac{a}{\|a\|}, -1\right) : |\mu| = 1\right\}.
\]
Moreover, by Lemma~\ref{block matrix norm} $(z_n)$ is a 
norming sequence for $M_{a,I}$ if and only if it is a norming 
sequence for $M_{a,I}P$, where $P$ denotes the orthogonal projection 
onto $\mathcal{H}_0^\perp$.

\medskip

Next, for any $A \in \mathcal{B}(\mathcal{H} \oplus \mathbb{C})$, consider the 
collection 
\begin{multline*}
\Lambda(M_{\phi_{a,U}}, A) = \left\{\mu \in \mathbb{C} : 
\lim_n \langle Az_n, y_n\rangle = \mu,\ 
(z_n),(y_n) \subset S_{\mathcal{H}\oplus\mathbb{C}},\right.\\
\left.\hspace{4cm} \lim_n \langle M_{\phi_{a,U}}z_n, y_n\rangle 
= \|M_{\phi_{a,U}}\|\right\}.
\end{multline*}
We show this collection is a singleton set. Let $\lambda_0 \in 
\Lambda(M_{\phi_{a,U}}, A)$, realized by the sequences $(z_n)$ and $(y_n)$. 
Since $|\langle M_{\phi_{a,U}}z_n, y_n\rangle| \leq \|M_{\phi_{a,U}}z_n\|$, the 
condition $\lim_n\langle M_{\phi_{a,U}}z_n, y_n\rangle = \|M_{\phi_{a,U}}\|$ 
forces $(z_n)$ to be a norming sequence for $M_{\phi_{a,U}}$, hence for 
$M_{a,I}$, and therefore, for $M_{a,I}P$.

\medskip

\noindent Since $\mathcal{H}_0^\perp$ is 
two-dimensional, its closed unit ball is a compact subset of $B_{\H\oplus \mathbb{C}}$. We observe that 
$\lim_n\|(I-P)z_n\| = 0$ (by the same argument as in 
Lemma~\ref{block matrix norm}), and the sequence $(Pz_n)$ lies in the closed unit ball of $\mathcal{H}_0^\perp$. We may therefore pass to a 
subsequence and assume $Pz_{n_k} \to z_0$ for some unit vector $z_0$ in $\mathcal{H}_0^\perp$. Since $(I-P)z_{n_k} \to 0$, we get 
$z_{n_k} \to z_0$, and we know from the norm attainment set of $M_{a,I}$ that the unit vector $z_0$ is uniquely determined by:
\[
z_0 = \mu_0\cdot\frac{1}{\sqrt{2}}\left(\frac{a}{\|a\|}, -1\right)
\quad\text{for some unimodular scalar } \mu_0.
\]

\noindent A parallel argument applies to $(y_n)$. Since $M_{a,I}$ 
is self-adjoint and $\lim_n\langle M_{\phi_{a,U}}z_n, y_n\rangle = 
\|M_{\phi_{a,U}}\| \in \mathbb{R}$, we have
\[
\lim_n\langle M_{\phi_{a,U}}z_n, y_n\rangle 
= \lim_n\langle M_{a,I}z_n, \widetilde{U}^*y_n\rangle 
= \lim_n\langle M_{a,I}\widetilde{U}^*y_n, z_n\rangle 
= \|M_{\phi_{a,U}}\|,
\]
which shows $(\widetilde{U}^*y_n)$ is also a norming sequence for 
$M_{a,I}$. By the same compactness argument, passing to a further 
subsequence gives $\widetilde{U}^*y_{n_k} \to y_0$ where
\[
y_0 = \sigma_0\cdot\frac{1}{\sqrt{2}}\left(\frac{a}{\|a\|}, -1\right)
\quad\text{for some unimodular scalar } \sigma_0.
\]

\noindent By a diagonal argument we extract a common subsequence 
$(n_p)$ along which both $z_{n_p} \to \mu_0\tilde{z}_0$ and 
$\widetilde{U}^*y_{n_p} \to \sigma_0\tilde{z}_0$, where $\mu_0$ and $\sigma_0$ are unimodular scalars and $\tilde{z}_0 = \frac{1}{\sqrt{2}}\left(a/\|a\|, -1\right)$. 
The condition $\lim_p\langle M_{\phi_{a,U}}z_{n_p}, y_{n_p}\rangle = 
\|M_{\phi_{a,U}}\|$ then gives
\[
\mu_0\overline{\sigma_0}(1 + \|a\|) = \|M_{\phi_{a,U}}\|,
\]
which forces $\mu_0\overline{\sigma_0} = 1$, hence $\mu_0 = \sigma_0$.

\noindent Therefore $\lambda_0$ is completely determined:
\[
\lambda_0 = \lim_p\langle Az_{n_p}, y_{n_p}\rangle 
= \left\langle A\tilde{z}_0,\ \widetilde{U}\tilde{z}_0\right\rangle
= \frac{1}{2}\left\langle A\!\left(\tfrac{a}{\|a\|}, -1\right),\ 
\widetilde{U}\!\left(\tfrac{a}{\|a\|}, -1\right)\right\rangle.
\]
Since this value depends only on $A$ and not on the choice of 
sequences, $\Lambda(M_{\phi_{a,U}}, A)$ is a singleton set for every 
$A \in \mathcal{B}(\mathcal{H} \oplus \mathbb{C})$. By 
Corollary~\ref{smoothness-Hilbert}, $M_{\phi_{a,U}}$ is a smooth point 
in $\mathcal{B}(\mathcal{H} \oplus \mathbb{C})$, completing the 
proof.
\end{proof}

\subsection{Orthogonality}\hfill\\

\noindent The same block matrix structure that identifies smooth 
points also translates orthogonality into a 
purely geometric condition on $\partial B$, which is Theorem \ref{ortho: aut}

\begin{proof}[Proof of Theorem \ref{ortho: aut}]
We first prove the necessity. Suppose $M_{\phi_{a,U}} \perp_B M_{\phi_{b,V}}$. 
By Lemma~\ref{BS-Theorem}, there exists a norming sequence $(z_n)$ 
for $M_{\phi_{a,U}}$ such that
\[
\lim_n \langle M_{\phi_{a,U}}z_n, M_{\phi_{b,V}}z_n\rangle = 0.
\]
By the same compactness argument of Theorem~\ref{smooth: Aut}, we pass 
to a subsequence $(z_{n_k})$ of $(z_n)$  and obtain
\[
z_{n_k} \to z_0 = \frac{\mu_0}{\sqrt{2}}
\left(\frac{a}{\|a\|}, -1\right)
\quad\text{for some unimodular scalar } \mu_0.
\]

\noindent The next step is to express $M_{\phi_{a,U}}z_0$ and $M_{\phi_{b,V}}z_0$ 
in terms of the M\"{o}bius maps $\phi_{a,U}$ and $\phi_{b,V}$ 
evaluated at $-a/\|a\|$. A direct computation using the 
block matrix structure gives
\begin{align*}
M_{\phi_{a,U}}\left(\frac{a}{\|a\|}, -1\right) 
&= -(1 + \|a\|)
\left(\phi_{a,U}\!\left(-\tfrac{a}{\|a\|}\right), 1\right),\\
M_{\phi_{b,V}}\left(\frac{a}{\|a\|}, -1\right) 
&= -\left(1 + \left\langle\frac{a}{\|a\|}, b\right\rangle\right)
\left(\phi_{b,V}\!\left(-\tfrac{a}{\|a\|}\right), 1\right).
\end{align*}
Note that $1 + \|a\| > 0$ and 
$1 + \langle a/\|a\|, b\rangle \neq 0$ since 
$|\langle a/\|a\|, b\rangle| \leq \|b\| < 1$.

\medskip

\noindent By continuity of the inner product, passing to the limit 
along the subsequence gives
\begin{align*}
0 &= \lim_k\langle M_{\phi_{a,U}}z_{n_k}, M_{\phi_{b,V}}z_{n_k}\rangle 
= \langle M_{\phi_{a,U}}z_0, M_{\phi_{b,V}}z_0\rangle\\
&= \frac{1}{2}(1+\|a\|)
\left(1+\left\langle\frac{a}{\|a\|},b\right\rangle\right)
\left[\left\langle\phi_{a,U}\!\left(-\tfrac{a}{\|a\|}\right),
\phi_{b,V}\!\left(-\tfrac{a}{\|a\|}\right)\right\rangle 
+ 1\right].
\end{align*}
Cancelling both the factors, we conclude
\[
\left\langle\phi_{a,U}\!\left(-\tfrac{a}{\|a\|}\right),\, 
\phi_{b,V}\!\left(-\tfrac{a}{\|a\|}\right)\right\rangle = -1.
\]
Since both images lie on $\partial B$, the equality case of the 
Cauchy--Schwarz inequality now forces
\[
\phi_{a,U}\!\left(-\tfrac{a}{\|a\|}\right) 
= -\phi_{b,V}\!\left(-\tfrac{a}{\|a\|}\right),
\]
that is, $\phi_{a,U}\!\left(-a/\|a\|\right)$ and 
$\phi_{b,V}\!\left(-a/\|a\|\right)$ are antipodal points 
on $\partial B$.

\medskip

\noindent We now prove the sufficiency. Suppose $\phi_{a,U}\!\left(-a/\|a\|\right)$ 
and $\phi_{b,V}\!\left(-a/\|a\|\right)$ are antipodal on 
$\partial B$, so that
\[
\left\langle\phi_{a,U}\!\left(-\tfrac{a}{\|a\|}\right),\, 
\phi_{b,V}\!\left(-\tfrac{a}{\|a\|}\right)\right\rangle + 1 = 0.
\]
Let $w = \frac{1}{\sqrt{2}}\left(a/\|a\|, -1\right)$, which 
is the unit vector at which $M_{\phi_{a,U}}$ attains its norm. For any 
scalar $\lambda$, the block matrix computations above give
\begin{align*}
\|(M_{\phi_{a,U}}+\lambda M_{\phi_{b,V}})w\|^2 
&= \|M_{\phi_{a,U}}w\|^2 
+ 2\,\mathrm{Re}\,\lambda\langle M_{\phi_{b,V}}w, M_{\phi_{a,U}}w\rangle 
+ |\lambda|^2\|M_{\phi_{b,V}}w\|^2.
\end{align*}
The cross term vanishes by the antipodality condition:
\begin{align*}
\langle M_{\phi_{a,U}}w, M_{\phi_{b,V}}w\rangle 
& = \frac{(1+\|a\|)\left(1+\left\langle a/\|a\|,b\right\rangle\right)}{2}
\left[\left\langle\phi_{a,U}\!\left(-\tfrac{a}{\|a\|}\right),
\phi_{b,V}\!\left(-\tfrac{a}{\|a\|}\right)\right\rangle 
+ 1\right]\\
& = 0.
\end{align*}
Therefore,
\[
\|(M_{\phi_{a,U}}+\lambda M_{\phi_{b,V}})w\|^2 
= \|M_{\phi_{a,U}}w\|^2 + |\lambda|^2\|M_{\phi_{b,V}}w\|^2 
\geq \|M_{\phi_{a,U}}w\|^2 = \|M_{\phi_{a,U}}\|^2,
\]
which shows $\|M_{\phi_{a,U}}+\lambda M_{\phi_{b,V}}\| \geq \|M_{\phi_{a,U}}\|$ for 
all scalars $\lambda$. By definition, $M_{\phi_{a,U}} \perp_B M_{\phi_{b,V}}$ 
in $\mathcal{B}(\mathcal{H}\oplus\mathbb{C})$, completing the proof.
\end{proof}

\noindent The antipodality condition admits an explicit 
formulation directly in terms of the parameters of the 
automorphisms.

\begin{corollary}\label{cor: ortho}
Let $M_{\phi_{a,U}}\perp_B M_{\phi_{b,V}}$ in $M \left(\H \right)$ if and only if 
\[
U \left(-\frac{a}{\|a\|} \right)=-V\phi_{b,I} \left( \left(-\frac{a}{\|a\|} \right) \right).
\]
\end{corollary}

\begin{proof}
We first observe that $-a/\|a\|$ is a fixed point of 
$\phi_{a,I}$, since
\[
\phi_{a,I}\left(-\frac{a}{\|a\|}\right) 
= \frac{-\frac{a}{\|a\|} - a}{1 + \|a\|} 
= -\frac{a}{\|a\|},
\]
and that $\phi_{b,V} = V\phi_{b,I}$. Therefore,
\begin{align*}
M_{\phi_{a,U}} \perp_B M_{\phi_{b,V}} 
&\iff \phi_{a,U}\!\left(-\tfrac{a}{\|a\|}\right) 
= -\phi_{b,V}\!\left(-\tfrac{a}{\|a\|}\right)\\
&\iff -U\tfrac{a}{\|a\|} 
= -V\phi_{b,I}\!\left(-\tfrac{a}{\|a\|}\right)\\
&\iff U\!\left(-\tfrac{a}{\|a\|}\right) 
= -V\phi_{b,I}\!\left(-\tfrac{a}{\|a\|}\right),
\end{align*}
and the proof is complete.
\end{proof}

\begin{corollary}
For any unitary operator $U$ on $\mathcal{H}$ and any 
$a \in B \setminus \left\{0\right\}$, $M_{\phi_{a,U}} \perp_B M_{a,-U}$ in 
$M(\mathcal{H})$.
\end{corollary}

\begin{proof}
Since $-a/\|a\|$ is a fixed point of $\phi_{a,I}$, we have
\[
\phi_{a,U}\!\left(-\tfrac{a}{\|a\|}\right) = -U\tfrac{a}{\|a\|} 
= -\phi_{a,-U}\!\left(-\tfrac{a}{\|a\|}\right),
\]
so the two images are antipodal on $\partial B$. The result follows 
from Theorem~\ref{ortho: aut}.
\end{proof}

\section{Symmetry and preservation of orthogonality in $M(\H)$}\label{section:preservation}

\noindent We now turn to the symmetry properties of 
Birkhoff--James orthogonality in $M(\mathcal{H})$. The key tool in both proofs is the following lemma, which produces an explicit unitary operator that maps one boundary point to another boundary point through Householder reflector.

\begin{lemma}\label{lem:householder}
Let $y, w \in \partial B$. Then there exists a unitary operator 
$V$ on $\mathcal{H}$ such that $Vw = y$.
\end{lemma}

\begin{proof}
The proof is trivial in dimension $1$, since $z\mapsto y\overline{w}z$ is the required unitary. If $dim~H\geq 2$, we consider the Householder reflector: if $y = w$, take $V = I$. Otherwise, set 
$\mu_0 = \mathrm{sgn}\langle w, y\rangle$ and 
$z_0 = \frac{w - \mu_0 y}{\|w - \mu_0 y\|}$, and define
\[
H = I - 2(z_0\otimes z_0).
\]
Since $\langle w, \mu_0 y\rangle = \overline{\mathrm{sgn}
\langle w,y\rangle}\langle w,y\rangle$ is a positive real 
number, we have $\|w - \mu_0 y\|^2 = 2(1 - \langle w, 
\mu_0 y\rangle)$, and a direct computation gives
\[
H(w) = w - 2\frac{1 - \langle w, \mu_0 y\rangle}
{2(1 - \langle w, \mu_0 y\rangle)}(w - \mu_0 y) = \mu_0 y.
\]
The operator $H$ is self-adjoint and unitary by definition, 
and therefore $V = \overline{\mu_0}H$ satisfies 
$Vw = \overline{\mu_0}\mu_0 y = y$.
\end{proof}

\noindent We next prove Theorem \ref{thm:left} which presents a characterization left-symmetric and right-symmetric points of $M(\H)$.

\begin{proof}[Proof of Theorem \ref{thm:left}]
We divide the proof into two parts.

\medskip

\noindent {\bf Part-I: No left-symmetric points}

\medskip

\noindent \textit{Case I: Non-rotations are not left-symmetric.}

\medskip

\noindent Let $M_{\phi_{a,U}} \in M(\mathcal{H})$ 
for some nonzero $a \in B$. We construct 
$M_{\phi_{b,V}} \in M(\mathcal{H})$ such that $M_{\phi_{a,U}} \perp_B M_{\phi_{b,V}}$ 
but $M_{\phi_{b,V}} \not\perp_B M_{\phi_{a,U}}$, showing $M_{\phi_{a,U}}$ is not 
left-symmetric.

\medskip

\noindent We choose $b \in B$ with the following specification:
\[
Im~(\langle a, b\rangle^2)\neq 0, \quad \|a\|\neq \|b\|.
\]
Existence of such an element $b$ is always ensured in any complex Hilbert space. By Corollary~\ref{cor: ortho}, $M_{\phi_{a,U}} \perp_B M_{\phi_{b,V}}$ 
if and only if $V\!\left(-\phi_{b,I}\!\left(-a/\|a\|\right)
\right) = U\!\left(-a/\|a\|\right)$.

\medskip

\noindent Next, we specify $V$. Setting 
$y = U\!\left(-a/\|a\|\right)$ and 
$w = -\phi_{b,I}\!\left(-a/\|a\|\right)$, both of which 
lie on $\partial B$, Lemma~\ref{lem:householder} produces a 
unitary $V$ satisfying $Vw = y$, so $M_{\phi_{a,U}} \perp_B M_{\phi_{b,V}}$ 
by construction.

\medskip

\noindent It remains to show $M_{\phi_{b,V}} \not\perp_B M_{\phi_{a,U}}$. Suppose for contradiction that $M_{\phi_{b,V}} \perp_B M_{\phi_{a,U}}$. By Corollary~\ref{cor: ortho}, $M_{\phi_{b,V}} \perp_B M_{\phi_{a,U}}$ would 
require $V\!\left(-b/\|b\|\right) = 
-U\phi_{a,I}\!\left(-b/\|b\|\right)$. 
Since $V$ and $U$ are unitary operators, using 
\[
V\!\left(-\phi_{b,I}\!\left(-\frac{a}{\|a\|}\right)\right) 
= U\!\left(-\frac{a}{\|a\|}\right), \quad 
V\!\left(-\frac{b}{\|b\|}\right) 
= -U\phi_{a,I}\!\left(-\frac{b}{\|b\|}\right),
\]
we get
\begin{align*}
\left\langle 
\phi_{b,I}\!\left(-\tfrac{a}{\|a\|}\right),\, 
\tfrac{b}{\|b\|}
\right\rangle 
&= \left\langle 
V\!\left(-\phi_{b,I}\!\left(-\tfrac{a}{\|a\|}\right)\right),\, 
V\!\left(-\tfrac{b}{\|b\|}\right)
\right\rangle \\
&= \left\langle 
U(-\tfrac{a}{\|a\|}),\, 
-U\phi_{a,I}\!\left(-\tfrac{b}{\|b\|}\right)
\right\rangle.\\
&= \left\langle 
\tfrac{a}{\|a\|},\, 
\phi_{a,I}\!\left(-\tfrac{b}{\|b\|}\right)
\right\rangle.
\end{align*}
Using 
$P_b(a/\|a\|) = \frac{\langle a,b\rangle}{\|a\|\|b\|^2}b$ 
and $Q_b(a/\|a\|) \perp b$, the left-hand side gives
\[
\left\langle 
\phi_{b,I}\!\left(-\tfrac{a}{\|a\|}\right),\, 
\tfrac{b}{\|b\|}
\right\rangle 
= \frac{-\dfrac{\langle a,b\rangle}{\|a\|\|b\|} - \|b\|}
{1 + \dfrac{\langle a,b\rangle}{\|a\|}}.
\]
with the role of $a$ and $b$ interchanged, the right-hand side becomes
\[
\left\langle 
\tfrac{a}{\|a\|},\, 
\phi_{a,I}\!\left(-\tfrac{b}{\|b\|}\right)
\right\rangle 
= \frac{-\dfrac{\langle a,b\rangle}{\|a\|\|b\|} - \|a\|}
{1 + \dfrac{\langle a,b\rangle}{\|b\|}}.
\]
Equating the two sides, we get 
\[
-\frac{\langle a,b\rangle^2}{\|a\|\|b\|^2} - \|b\| 
= -\frac{\langle a,b\rangle^2}{\|a\|^2\|b\|} - \|a\|,
\]
which rearranges to
\[
\frac{\langle a,b\rangle^2\bigl(\|b\| - \|a\|\bigr)}
{\|a\|^2\|b\|^2} = \|b\| - \|a\|.
\]
Since $\|a\| \neq \|b\|$, we may divide both sides by 
$\|b\| - \|a\|$ to obtain 
$\langle a,b\rangle^2 = \|a\|^2\|b\|^2$, contradicting the choice of $b$ with 
$\mathrm{Im}(\langle a,b\rangle)^2 \neq 0$. Therefore $M_{\phi_{a,U}}$ is not left-symmetric.

\medskip

\noindent \textit{Case II: Pure rotations are not left-symmetric.}

\medskip

\noindent If $dim ~\H=1$, i.e., $\H=\mathbb{C}$, consider a pure rotation $M_{0,e^{i\theta}}$ and let 
$b \in B \cap \mathbb{R}$ with $b \neq 0$. Set 
$e^{i\phi} = -e^{i\theta}\cdot\frac{1+ib}{1-ib}$, and 
note that $e^{i\phi} \neq -e^{i\theta}$ since 
$b \neq 0$. Taking the unit vector 
$z = \frac{1}{\sqrt{2}}(1, i) \in \mathbb{C}^2$:
\begin{align*}
\left\langle M_{0,e^{i\theta}}^*M_{b,e^{i\phi}}z,\, 
z\right\rangle
&= \frac{1}{2}\!\left[e^{i(\phi-\theta)}(1-ib) 
+ (1+ib)\right]\\
& = \frac{1}{2}\!\left[-\frac{1+ib}{1-ib}(1-ib) 
+ (1+ib)\right]\\
& = \frac{1}{2}\!\left[-(1+ib) + (1+ib)\right]\\
& = 0.
\end{align*}
Since $M_{0,e^{i\theta}}$ is unitary, $z$ is a norming 
vector for $M_{0,e^{i\theta}}$. Thus, $M_{0,e^{i\theta}} \perp_B M_{b,e^{i\phi}}$ by 
Lemma~\ref{BS-Theorem}. By Theorem~\ref{ortho: aut}, 
$M_{b,e^{i\phi}} \perp_B M_{0,e^{i\theta}}$ requires 
antipodality of $\phi_{b,e^{i\phi}}(-b/|b|) 
= -e^{i\phi}b/|b|$ and 
$\phi_{0,e^{i\theta}}(-b/|b|) = -e^{i\theta}b/|b|$, 
which gives $e^{i\phi} = -e^{i\theta}$. However, this fails, since
$e^{i\phi} = -e^{i\theta}\cdot\frac{1+ib}{1-ib} 
\neq -e^{i\theta}$ for $b \neq 0$. 
Therefore, $M_{0,e^{i\theta}}$ is not left-symmetric.

\bigskip

\noindent If $dim ~ \H \geq 2$, consider a pure rotation $M_{0,U}$, and let $a \in B$ be any nonzero element. Set 
$u = U\!\left(a/\|a\|\right)$ and choose $w \in 
\partial B$ with $w \perp u$. By Lemma~\ref{lem:householder}, 
there exists a unitary $V$ with $V\!\left(a/\|a\|\right) 
= w$. Since $M_{0,U}$ is unitary, every unit vector is a norming 
vector, and since
\[
\left\langle M_{0,U}\!\left(\tfrac{a}{\|a\|},0\right), 
M_{a,V}\!\left(\tfrac{a}{\|a\|},0\right)\right\rangle 
= \left\langle u, w\right\rangle = 0,
\]
we conclude $M_{0,U} \perp_B M_{a,V}$. However, by 
Corollary~\ref{cor: ortho},
\[
-U\!\left(\phi_{0,I}\!\left(-\tfrac{a}{\|a\|}\right)\right) 
= U\!\left(\tfrac{a}{\|a\|}\right) = u \neq -w 
= V\!\left(-\tfrac{a}{\|a\|}\right),
\]
since $\langle u, w\rangle = 0$.
so $M_{a,V} \not\perp_B M_{0,U}$. Therefore pure rotations 
are not left-symmetric points.

\medskip

\noindent\textbf{Part 2: Right-symmetric points are 
precisely pure rotations.}

\medskip

\noindent\textit{Case I: Pure rotations are right-symmetric.}

\medskip

\noindent Let $M_{0,U}$ be a pure rotation and suppose 
$M_{a,V} \perp_B M_{0,U}$ for some 
$M_{a,V} \in M(\mathcal{H})$. If $a = 0$, both 
$M_{0,V}$ and $M_{0,U}$ are unitary and the argument 
is symmetric by Lemma~\ref{BS-Theorem}. If $a \neq 0$, By 
Theorem~\ref{ortho: aut}, $\phi_{a,V}\!\left(-a/\|a\|
\right)$ and $U\!\left(-a/\|a\|\right)$ are antipodal, 
so
\[
\left\langle M_{a,V}\!\left(-\tfrac{a}{\|a\|},1\right), 
M_{0,U}\!\left(-\tfrac{a}{\|a\|},1\right)\right\rangle = 0,
\]
so the constant norming sequence 
$x_n \equiv \frac{1}{\sqrt{2}}\!\left(-a/\|a\|, 1\right)$ 
satisfies 
$\langle M_{0,U}x_n, M_{a,V}x_n\rangle 
= \frac{1}{2}\langle M_{a,V}(-a/\|a\|,1), 
M_{0,U}(-a/\|a\|,1)\rangle = 0$.
Since $M_{0,U}$ is unitary, it attains its norm at every unit 
vector, and by Lemma~\ref{BS-Theorem} 
immediately gives $M_{0,U} \perp_B M_{a,V}$. This shows every 
pure rotation is right-symmetric.

\medskip

\noindent\textit{Case II: Non-rotations are not right-symmetric.}

\medskip

\noindent If $dim ~ \H=1$, i.e., $\H = \mathbb{C}$, consider $M_{a,e^{i\theta}} \in M(\mathbb{C})$ with 
$a \neq 0$. Choose $b \in B$ with $b \neq 0$, 
$|b| \neq |a|$, and $b/|b| \neq \pm\, a/|a|$. 
Since $\phi_{a,e^{i\theta}}(-b/|b|)$ is unimodular 
and $-b/|b|$ is a fixed point of $\phi_{b,I}$, 
define $e^{i\phi}$ so that 
\[
e^{i\phi}\frac{b}{|b|} 
= e^{i\theta}\phi_{a,I}\!\left(-\frac{b}{|b|}\right).
\tag{I}
\]
which satisfies the antipodality condition, and By Theorem~\ref{ortho: aut}, we have 
$M_{b,e^{i\phi}} \perp_B M_{a,e^{i\theta}}$.

\medskip

\noindent We claim $M_{a,e^{i\theta}} \not\perp_B 
M_{b,e^{i\phi}}$. Suppose for contradiction that 
$M_{a,e^{i\theta}} \perp_B M_{b,e^{i\phi}}$. By 
Theorem~\ref{ortho: aut}, together with the fixed 
point identity 
$\phi_{a,e^{i\theta}}(-a/|a|) = -e^{i\theta}a/|a|$, 
this gives:
\[
e^{i\theta}\frac{a}{|a|} 
= e^{i\phi}\phi_{b,I}\!\left(-\frac{a}{|a|}\right).
\tag{II}
\]
Multiplying $\mathrm{(I)}$ and $\mathrm{(II)}$ and 
cancelling $e^{i(\theta+\phi)} \neq 0$, we get
\[
\frac{a}{|a|}\cdot\frac{b}{|b|} 
= \phi_{b,I}\!\left(-\frac{a}{|a|}\right)\cdot
\phi_{a,I}\!\left(-\frac{b}{|b|}\right). \tag{III}
\]
Using the direct computations:
\[
\phi_{b,I}\!\left(-\frac{a}{|a|}\right) 
= \frac{-\frac{b}{|b|}\!\left(\frac{a}{|a|}
+|b|\frac{b}{|b|}\right)}{\frac{b}{|b|}
+|b|\frac{a}{|a|}}, \quad
\phi_{a,I}\!\left(-\frac{b}{|b|}\right) 
= \frac{-\frac{a}{|a|}\!\left(\frac{b}{|b|}
+|a|\frac{a}{|a|}\right)}{\frac{a}{|a|}
+|a|\frac{b}{|b|}},
\]
equation $\mathrm{(III)}$ becomes:
\[
\frac{\!\left(\frac{a}{|a|}+|b|\frac{b}{|b|}\right)
\!\left(\frac{b}{|b|}+|a|\frac{a}{|a|}\right)}
{\!\left(\frac{b}{|b|}+|b|\frac{a}{|a|}\right)
\!\left(\frac{a}{|a|}+|a|\frac{b}{|b|}\right)} = 1.
\]
Expanding both sides and cancelling the common terms 
$\frac{a}{|a|}\cdot\frac{b}{|b|}$ and 
$|a||b|\cdot\frac{a}{|a|}\cdot\frac{b}{|b|}$:
\[
(|a|-|b|)\!\left[\left(\frac{a}{|a|}\right)^2 
- \left(\frac{b}{|b|}\right)^2\right] = 0,
\]
which contradicts $|a| \neq |b|$ and 
$a/|a| \neq \pm\, b/|b|$. Therefore 
$M_{a,e^{i\theta}} \not\perp_B M_{b,e^{i\phi}}$, 
and $M_{a,e^{i\theta}}$ is not right-symmetric.

\medskip

\noindent If $dim ~\H \geq 2$, consider $M_{\phi_{a,U}} \in M(\mathcal{H})$ with $a \neq 0$. 
Set $u = U\!\left(a/\|a\|\right)$ and choose a unit 
vector $w \perp u$. By Lemma~\ref{lem:householder}, there exists 
a unitary $V$ with $V\!\left(a/\|a\|\right) = w$. Since
\[
\left\langle M_{0,V}\!\left(\tfrac{a}{\|a\|},0\right), 
M_{\phi_{a,U}}\!\left(\tfrac{a}{\|a\|},0\right)\right\rangle 
= \left\langle (w,0),\, (u, -\|a\|)\right\rangle 
= \langle w, u\rangle = 0,
\]
and $M_{0,V}$ is unitary, we get $M_{0,V} \perp_B M_{\phi_{a,U}}$. 
However, by Corollary~\ref{cor: ortho},
\[
U\!\left(-\tfrac{a}{\|a\|}\right) = -u \neq w 
=- V\!\left(\phi_{0,I}\!\left(-\tfrac{a}{\|a\|}\right)\right),
\]
since $u \perp w$. Thus, $M_{\phi_{a,U}} \not\perp_B M_{0,V}$, showing $M_{\phi_{a,U}}$ is not 
right-symmetric.

\medskip

\noindent This completes the proof.
\end{proof}

As an immediate corollary, we get the following.

\begin{corollary}
$M(\H)$ has no symmetric points.
\end{corollary}

\noindent As noted in the introduction, orthogonality relation in $M(\H)$ naturally induces a orthogonality relation $\perp_{Aut}$. We now characterize the inner automorphisms that preserves $\perp_{Aut}$.

\begin{proposition}
Let $\Phi:Aut(B)\to Aut(B)$ be a bijective map that preserves $\perp_{Aut}$ in both directions. Then $\Phi$ preserves pure rotations.
\end{proposition}

\begin{proof}
Let $\phi_{0,U}$ be a pure rotation for some unitary operator $U$ on $\H$. Let $\Phi(\phi_{0,U})=\phi_{b,V}$ for some $b\in B$ and unitary operator $V$ on $\H$. We observe that for any $\phi_{c,W}\in Aut(B)$,  
\[
\phi_{c,W}\perp_{Aut} \Phi(\phi_{0,U}) \iff \Phi(\phi_{0,U})\perp_{Aut} \phi_{c,W}
\]
since $M_{\phi_{0,U}}$ is a right-symmetric point in $M(\H)$. Consequently, $\phi_{b,V}$ is a right-symmetric point in $M(\H)$, and by Theorem \ref{thm:left} we have $b=0$. Therefore, $\Phi$ preserves pure rotation.
\end{proof}

\begin{proof}[Proof of Theorem~\ref{thm: preserve}]
An inner automorphism of $\Aut(B)$ has the form $\Phi_\psi:\phi\mapsto\psi\circ\phi\circ\psi^{-1}$ for some $\psi\in\Aut(B)$.
 
\medskip
 
\noindent We first prove the necessity. Suppose $\Phi_\psi$ preserves $\perp_{\Aut}$ in both directions. Then by the preceding proposition, $\psi$ preserves pure rotations, meaning $\psi\circ\phi_{0,U}\circ\psi^{-1}$ is a pure rotation for every unitary operator $U$ on $\H$. We show $\psi$ must be a pure rotation. Write $\psi = \phi_{c,W}$ for some $c\in B$ and unitary $W$. Since $\phi_{c,W}^{-1} = \phi_{-c,I}\circ W^{-1}$, we compute:
\[
(\phi_{c,W}\circ\phi_{0,U}\circ\phi_{c,W}^{-1})(0)
= \phi_{c,W}(\phi_{0,U}(c))
= \phi_{c,W}(Uc).
\]
For this to be a pure rotation, its value at $0$ must be $0$. Since $\phi_{c,W}(z) = 0$ if and only $z = c$, the condition $\phi_{c,W}(Uc) = 0$ requires $Uc = c$ for every unitary $U$. But $Uc = c$ for all unitaries $U$ forces $c = 0$. Consequently, $\psi = \phi_{0,W}$ is a pure rotation.
 
\medskip
 
\noindent We now prove the sufficiency. Suppose $\psi = \phi_{0,W}$ is a pure rotation. For any $\phi_{a,U}\in\Aut(B)$, conjugation by $W$ gives:
\[
\phi_{0,W}\circ\phi_{a,U}\circ\phi_{0,W}^{-1} = \phi_{Wa,\,WUW^{-1}}.
\]

\medskip
 
\noindent By Theorem~\ref{ortho: aut}, $M_{\phi_{a,U}}\perp_B M_{\phi_{b,V}}$ if and only if $\phi_{a,U}(-a/\|a\|)$ and $\phi_{b,V}(-a/\|a\|)$ are antipodal. Therefore, to show $\Phi_\psi$ preserves $\perp_{Aut}$, we need to prove that 
\[
\phi_{Wa,WUW^{-1}}(-Wa/\|a\|)=-\phi_{Wb,WVW^{-1}}(-Wa/\|a\|).
\]
Using the identity 
\[
\phi_{Wc,WRW^{-1}}(Wz) = W\phi_{c,R}(z), \quad z\in B,
\]
for any $c\in B$, and unitary operator $R$ on $\H$, and $z\in B$, we get
\[
\phi_{Wa,WUW^{-1}}\!\left(-\tfrac{Wa}{\|a\|}\right) = W\phi_{a,U}\!\left(-\tfrac{a}{\|a\|}\right),
\quad
\phi_{Wb,WVW^{-1}}\!\left(-\tfrac{Wa}{\|a\|}\right) = W\phi_{b,V}\!\left(-\tfrac{a}{\|a\|}\right).
\]
Since $W$ is unitary, it preserves antipodality. Since the antipodality condition is equivalent in both directions, similar argument shows $\Phi_\psi^{-1}$ preserves $\perp_{Aut}$. Thus, $\Phi_\psi$ preserves $\perp_{Aut}$ in both directions, and the proof is now complete.
\end{proof}

\noindent Theorem \ref{thm: preserve} remains open in context of general bijections that preserves $\perp_{Aut}$ in both directions.

\begin{remark}
The symmetry of Birkhoff--James orthogonality has been extensively 
studied in the linear setting \cite{Mal-Paul-Sain-book, Sain-2017, Sain-normattain}. In $\mathcal{B}(\mathcal{H})$, it is known that the unit ball of $\B(\H)$ has no left-symmetric points, while the only right-symmetric points are isometries and co-isometries \cite{Turnsek, Komuro-Saito-Tanaka}. The set $M(\mathcal{H})$, 
by contrast, is a highly non-linear subset of $\mathcal{B}(\mathcal{H}\oplus\mathbb{C})$. Nevertheless, the symmetry pattern that emerges 
is almost parallel: $M(\mathcal{H})$ admits no left-symmetric 
points, and its right-symmetric points are precisely the pure rotations, 
which are the unitary elements of $M(\mathcal{H})$. This suggests a certain rigidity in the orthogonality structure that persists beyond the linear framework, and it would be interesting to investigate whether similar phenomena appear in other non-linear 
subsets of operator algebras arising naturally from complex function 
theory.
\end{remark}

\section{Hyperbolic metric induced by operator norm}\label{section:Hyperbolic metric}

\noindent We now show that the normalized $J$-unitary block matrices
recovers the hyperbolic geometry of the unit ball. We first record the norm formula, from the proof of Theorem~\ref{smooth: Aut},
\begin{equation}\label{Eqn:norm}
\|\widetilde{M}_{\phi_{a,U}}\| = \frac{1+\|a\|}{\alpha(a)} = \sqrt{\frac{1+\|a\|}{1-\|a\|}}=\sqrt{\frac{1+\|\phi_{a,U}(0)\|}{1-\|\phi_{a,U}(0)\|}}
\end{equation}

\noindent We remark a Schwarz-Pick type observation for normalized block matrices associated to the automorphism group.
\begin{remark}
Let $f\colon B\to B$ be a holomorphic self-map with $f(0)=0$. By the
Schwarz--Pick lemma for the Hilbert ball \cite{Goebel-Reich}, $\|f(a)\|
\leq\|a\|$ for all $a\in B$. Since $t\mapsto\sqrt{(1+t)/(1-t)}$ is
strictly increasing on $[0,1)$, this gives
\[
\|\widetilde{M}_{\phi_{f(a),I}}\|
= \sqrt{\frac{1+\|f(a)\|}{1-\|f(a)\|}}
\leq \sqrt{\frac{1+\|a\|}{1-\|a\|}}
= \|\widetilde{M}_{\phi_{a,I}}\|,
\]
so the operator norm of $\widetilde{M}_{\phi_{a,I}}$ decreases under
holomorphic self-maps of $B$ fixing the origin. Moreover, if equality
$\|\widetilde{M}_{\phi_{f(a_0),I}}\| = \|\widetilde{M}_{\phi_{a_0,I}}\|$
holds for some nonzero $a_0\in B$, then $\|f(a_0)\|=\|a_0\|$, and
Theorem~III.2.3 of \cite{FV} gives $\|f(\zeta a_0)\|=|\zeta|\|a_0\|$
for all $|\zeta|<1/\|a_0\|$, hence
\[
\|\widetilde{M}_{\phi_{f(\zeta a_0),I}}\|
= \sqrt{\frac{1+|\zeta|\|a_0\|}{1-|\zeta|\|a_0\|}}
= \|\widetilde{M}_{\phi_{\zeta a_0,I}}\|
\quad\text{for all } |\zeta| < \frac{1}{\|a_0\|}.
\]
\end{remark}

The following algebraic identity, derived from the $J$-unitary structure, is the key ingredient.
\begin{lemma}\label{lem:sym}
For any $a,b\in B$,
\begin{equation}\label{eq:sym}
1-\|\phi_{a,I}(b)\|^2 = \frac{(1-\|a\|^2)(1-\|b\|^2)}{|1-\langle b,a\rangle|^2}.
\end{equation}
\end{lemma}
 
\begin{proof}
Apply the $J$-unitary identity $\widetilde{M}_{\phi_{a,U}}^*J\widetilde{M}_{\phi_{a,U}} = J$ to the vector $(b,1)\in\mathcal{H}\oplus\mathbb{C}$:
\[
\left\langle J\widetilde{M}_{\phi_{a,U}}\begin{pmatrix}
    b\\
    1
\end{pmatrix}, \widetilde{M}_{\phi_{a,U}}\begin{pmatrix}
    b\\
    1
\end{pmatrix}\right\rangle = \left\langle J\begin{pmatrix}
    b\\
    1
\end{pmatrix},\begin{pmatrix}
    b\\
    1
\end{pmatrix}\right\rangle = \|b\|^2-1.
\]
On the other hand,
\[
\left\langle J\widetilde{M}_{\phi_{a,U}}\begin{pmatrix}
    b\\
    1
\end{pmatrix}, \widetilde{M}_{\phi_{a,U}}\begin{pmatrix}
    b\\
    1
\end{pmatrix}\right\rangle =\frac{|1-\langle b, a \rangle|^2}{\alpha(a)^2}\left\langle J\begin{pmatrix}
   \phi_{a,U}(b)\\
    1
\end{pmatrix}, \begin{pmatrix}
   \phi_{a,U}(b)\\
    1
\end{pmatrix}\right\rangle 
\]
Equating and rearranging we get \ref{lem:sym}.
\end{proof}
 
\begin{proof}[Proof of Theorem~\ref{thm:hyp}]
 
\noindent\textit{Part (a): Submultiplicativity.}
 
\medskip
 
\noindent Recall that
\[
\widetilde{M}_{\phi_{a,U}}\odot \widetilde{M}_{\phi_{b,U}} = \widetilde{M}_{\phi_{a,U}\circ \phi_{b,U}}.
\]
Therefore, by (\ref{Eqn:norm}), we have
\begin{align*}
\|\widetilde{M}_{\phi_{a,U}}\odot\widetilde{M}_{\phi_{b,U}}\| & = \sqrt{\frac{1+\|\phi_{a,U}\circ \phi_{b,U}(0)\|}{1-\|\phi_{a,U}\circ \phi_{b,U}(0)\|}}\\
& = \sqrt{\frac{1+\|U\phi_{a,I}(-Ub)\|}{1-\|U\phi_{a,I}(-Ub)\|}}\\
& = \sqrt{\frac{1+\|\phi_{a,I}(-Ub)\|}{1-\|\phi_{a,I}(-Ub)\|}}
\end{align*}
By Cauchy--Schwarz, $|\langle b,a\rangle|\leq\|a\|\|b\|$, so $|1-\langle -Ub,a\rangle| = |1+\langle Ub,a\rangle|\leq 1+\|a\|\|b\|$. From Lemma~\ref{lem:sym} applied to $-Ub$ in place of $b$:
\[
1-\|\phi_{a,I}(-Ub)\|^2 \geq \frac{(1-\|a\|^2)(1-\|b\|^2)}{(1+\|a\|\|b\|)^2}.
\]
Since $(1+\|a\|\|b\|)^2-(1-\|a\|^2)(1-\|b\|^2) = (\|a\|+\|b\|)^2$, we get:
\[
\|\phi_{a,I}(-Ub)\| \leq \frac{\|a\|+\|b\|}{1+\|a\|\|b\|}.
\]
Since $x\mapsto\sqrt{(1+x)/(1-x)}$ is increasing:
\[
\|\widetilde{M}_{\phi_{a,U}}\odot\widetilde{M}_{\phi_{b,U}}\|
\leq \sqrt{\frac{(1+\|a\|)(1+\|b\|)}{(1-\|a\|)(1-\|b\|)}}
= \|\widetilde{M}_{\phi_{a,U}}\|\,\|\widetilde{M}_{\phi_{b,U}}\|. 
\]
 
\medskip

\noindent\textit{Part (b): $(F_U,d_U)$ is a metric space isometric to $(B,d_{\mathrm{hyp}})$.}
 
\medskip
 
\noindent This follows directly, from the following observation. Since $\phi_{b,U}^{-1} = \phi_{-b,I}\circ U^{-1}$, we compute $(\phi_{a,U}\circ\phi_{b,U}^{-1})(0) = \phi_{a,U}(b) = U\phi_{a,I}(b)$. Consequently,
\[
d_U(\widetilde{M}_{\phi_{a,U}},\widetilde{M}_{\phi_{b,U}})
= 2\ln\|\widetilde{M}_{\phi_{a,U}\circ\phi_{b,U}^{-1}}\|
= \ln\frac{1+\|\phi_{a,I}(b)\|}{1-\|\phi_{a,I}(b)\|} =  d_{\mathrm{hyp}}(a,b),
\]
by \cite{Goebel-Reich, FV}, which completes the proof.
\end{proof}
\noindent However, an independent proof can be obtained by using the submultiplicativity of norm proved in Part (a).

\begin{remark}
The positivity and symmetry can be verified directly. Triangle inequality follows from the submultiplicativity of norm. Indeed, for any $\widetilde{M}_{\phi_{a,U}},\widetilde{M}_{\phi_{b,U}},\widetilde{M}_{\phi_{c,U}}\in F_U$, we have
\begin{align*}
d_U(\widetilde{M}_{\phi_{a,U}},\widetilde{M}_{\phi_{b,U}}) & = 2~ln~\left\|\phi_{a,U}\circ\phi_{b,U}^{-1}\right\|\\
& = 2~ln~\left\|\phi_{a,U}\circ{\phi_{c,U}^{-1}}\circ \phi_{c,U}\circ \phi_{b,U}^{-1}\right\|\\
& \leq 2~ln~\left\|\phi_{a,U}\circ{\phi_{c,U}^{-1}}\|\|\phi_{c,U}\circ \phi_{b,U}^{-1}\right\|\\
& = 2~ln~\left\|\phi_{a,U}\circ{\phi_{c,U}^{-1}}\right\| + 2~ln~\left\|\phi_{c,U}\circ \phi_{b,U}^{-1}\right\|\\
& = d_U(\widetilde{M}_{\phi_{a,U}},\widetilde{M}_{\phi_{c,U}})
+ d_U(\widetilde{M}_{\phi_{c,U}},\widetilde{M}_{\phi_{b,U}}).
\end{align*}

\medskip

\noindent Now, the map $\Gamma: (F_U,d_U)\to(B,d_{\mathrm{hyp}})$ defined by $\Gamma(\widetilde{M}_{\phi_{a,U}}) = a$ gives the required isometry.
\end{remark}

\begin{remark}
The submultiplicativity in Part~(a) differs from the usual submultiplicativity of the operator norm $\|AB\|\leq\|A\|\|B\|$; here the product $\odot$ is defined via composition of automorphisms which differs from usual operator multiplication.
\end{remark}

\section{Data Availability}

No data is associated to the manuscript.

\section{Conflict of Interest}

Author declares no conflict of interest.

\end{document}